\documentclass[]{amsart}
\usepackage[textwidth=14cm,hcentering]{geometry}
\usepackage{amssymb}
\usepackage{amsthm}
\usepackage{amsmath,amscd}
\usepackage[mathscr]{euscript}
\usepackage[all]{xy}
\usepackage{float}
\usepackage[utf8]{inputenc}
\usepackage{lmodern}
\usepackage[T1]{fontenc}
\usepackage[textwidth=14cm,hcentering]{geometry}
\usepackage[colorlinks=true,linkcolor=red,citecolor=blue]{hyperref}
\usepackage{cleveref}
\usepackage{mathtools}
\usepackage{tikz}
\usepackage{bbm}
\usepackage{adjustbox}
\usepackage{todonotes}
\usepackage{pgf}

\usetikzlibrary{cd}
\usetikzlibrary{arrows,positioning}

\usepackage{amssymb}
\usepackage{amsthm}
\usepackage{amsmath,amscd}
\usepackage[mathscr]{euscript}
\usepackage[all]{xy}
\usepackage{float}
\usepackage[utf8]{inputenc}
\usepackage{lmodern}
\usepackage[T1]{fontenc}
\usepackage[textwidth=14cm,hcentering]{geometry}
\usepackage[colorlinks=true,linkcolor=red,citecolor=blue]{hyperref}
\usepackage{mathtools}
\usepackage{tikz}
\usepackage{bbm}
\usepackage{adjustbox}
\usepackage{todonotes}
\usepackage{pgf}
\usetikzlibrary{cd}
\usetikzlibrary{arrows,positioning}

\usepackage{hyperref}
\usepackage{cleveref}
\theoremstyle{definition}
\newtheorem{thm}{Theorem}[section]

\newtheorem{rmq}[thm]{Remark}
\newtheorem{prop}[thm]{Proposition}
\newtheorem{defn}[thm]{Definition}
\newtheorem{lemme}[thm]{Lemma}
\newtheorem{coroll}[thm]{Corollary}

\newtheorem{thmalpha}{Theorem} 

\newcommand{\hdim}{\mrm{hdim}}
\newcommand{\Emb}{\mrm{Emb}}
\newcommand{\Diff}{\mrm{Diff}}
\newcommand{\xto}[1]{\xrightarrow{#1}}

\newcommand{\xinj}{\xhookrightarrow}

\newcommand{\inj}{\hookrightarrow}

\newcommand{\bb}[1]{\mathbb{#1}}
\newcommand{\mc}[1]{\mathcal{#1}}
\newcommand{\msf}[1]{\mathsf{#1}}
\newcommand{\mbf}[1]{\mathbf{#1}}
\newcommand{\mrm}[1]{\mathrm{#1}}

\newcommand{\vide}{\varnothing}

\newcommand{\ol}{\overline}
\newcommand{\ul}{\underline}

\usepackage{mathtools}
\DeclareUnicodeCharacter{202F}{ }

\DeclareMathOperator{\colim}{colim}

\DeclareMathOperator{\conf}{Conf}

\DeclareMathOperator{\tfib}{tfib}

\DeclareMathOperator{\coker}{coker}

\newcommand{\FI}{\mathsf{FI}}
\newcommand{\FB}{\mathsf{FB}}
\usepackage{diagbox}
\title{Periodicity in the homology of moduli spaces of disconnected submanifolds}
\author{Nicolas Guès}

\begin{document}
	
	\maketitle

	\begin{abstract}
		We show that the moduli space of $n$ suitably embedded copies of a closed smooth manifold $P$ inside a closed smooth manifold $M$ satisfies cohomological periodicity over $\bb F_p$ when $n$ grows, with an explicit linear bound on the period and the periodicity range. This generalizes a known result about configuration spaces. We also show integral stability of the cohomology when $M$ is open, reproving a result of Palmer and improving the slope when inverting $2$. The main input in the proof is Goodwillie and Klein's multiple disjunction lemma for embedding spaces. As a corollary we get stability and periodicity results for some classes of symmetric diffeomorphism groups of manifolds.
	\end{abstract}
	\section*{Introduction}
 For $P$ a closed smooth manifold of handle dimension $p$ and $M$ a smooth manifold of dimension $m$, call $$\msf{Sub}_{e}(P \times n, M)$$ the moduli space of submanifolds of $M$ diffeomorphic to $P \times n$ with fixed isotopy class $e$ consisting of disjoint copies of $P$ inside a small ball inside $M$. The goal of this paper is to study the stability (resp. periodicity) of the cohomology of $\mrm{Sub}_e (P \times n, M )$ with coefficients in $\bb Z$ and $\bb Q$ (resp. $\bb F_\ell$). See \Cref{mainresult} for a precise statement. 

Using our method, we also recover the stabilization of homology for open $M$ when $p \leq \tfrac{m-2}{2}$, thereby improving a result of Palmer \cite{Palmer_2021} by weakening the assumption $\dim P \leq \frac {m-3} 2$ to an assumption on the handle dimension, which is always lower than the genuine dimension. In particular, the theorem applies to any manifold $P$ of dimension $ \frac{m-2} 2$. 

\subsection*{Homological Stability}
Homological stability  is a well-studied phenomenon in algebraic topology. One striking example is the homology of the unordered configuration spaces of an open manifold. These are defined as $\mc U \conf_n(M) = \Emb(\ul n, M)/S_n$ where $\Emb(P,M)$ is the space of embeddings between two manifolds $P$ and $M$, $\ul n$ is the discrete space with $n$ points and $S_n$ is the symmetric group on $n$ letters acting naturally by permuting the $n$ points. 
There is a well-defined operation of adding a point near the boundary of $M$ and this gives a homotopy class of maps $\mc U \conf_n(M) \to \mc U \conf_{n+1}(M)$. It was shown in several papers that the induced maps in homology stabilize when $n$ is big enough, with varying assumptions on the manifold and the coefficient ring (see in particular \cite{McDuff_1975}, \cite{Segal_1979},\cite{Randal_2013}).
The situation is different when $M$ is a closed manifold. In this case, stability over $\bb Z$ does not hold anymore. The simplest example is $\mc U \conf_n(S^2)$ whose first homology group is $\bb Z/ (2n-2)\bb Z$ and thus does not exhibit homological stability. This complication is mostly due to the fact that there is no coherent way to add a point when $M$ is closed. The only maps that make sense are those forgetting the points, but they exist only when we look at \emph{ordered} configurations. However, Church proved in \cite{Church_2012} that the \emph{rational} cohomology of configuration spaces of a closed manifold still stabilizes, using the new tools of $\FI$-modules and representation stability. Generalizing with torsion coefficients, Nagpal \cite{Nagpal_2015} proved that over $\bb F_\ell$, the cohomology is not stable but rather \emph{eventually periodic}, with period a power of $l$. 
A natural generalization is asking the same question for embedded submanifolds of higher dimension. This was studied by Martin Palmer \cite{Palmer_2021} in the case of an open background manifold $M$ and a manifold $P$ with or without boundary. More precisly he shows a homological stability result for moduli spaces of disconnected submanifolds of a fixed diffeomorphism type $P$ and isotopy class, where stabilisation occurs by adding new copies of $P$ in the boundary of $M$.\\

The goal of this article is to generalize this result using a different approach, inspired by representation stability. In particular our result applies when the background manifold is closed. In this case, as it happens for configuration spaces, we will show that the homology stabilizes over $\bb Q$ and is eventually periodic over $\bb F_\ell$. We also show that Palmer's range improves when inverting the prime $2$ in the case of an open manifold. Another application of this method is a new proof of homological stability for symmetric diffeomorphism groups of manifolds, obtained by \cite{Tillmann_2016} and connecting this result with Jimenez Rolland's analogous result for pure diffeomorphism groups \cite{JimenezRolland_2019}. Now we state the main result of the paper. Choose $P$ to be a closed manifold of handle dimension $p$ and $M$ a manifold (open or closed) of dimension $m$.  Choose an embedding $e : P \to M$ of $P$ inside a small ball in $M$. This gives rise to an embedding $e_n : P \times n \to M$ by stacking copies of $P$. Call $\mrm{Emb}_{e_n}(P \times n, M)$ the connected component of the embedding space of $P \times n$ inside $M$ containing $e_n$. This space has a free action of $\Diff(P^n)= \Diff(P) \wr S_n$. The quotient by this action is a component of the moduli space of submanifolds $\mrm{Sub}_{e_n}(P \times n, M)$.  
\begin{thmalpha}\label{mainresult}
	Suppose $p \leq \frac{m-2} 2$.
	\begin{enumerate}
		\item If $M$ is an open manifold, there are stabilization maps $$\mrm{Sub}_{e_n}(P \times n, M) \to \mrm{Sub}_{e_{n+1}}(P \times (n+1), M)$$ constructed by Palmer \cite{Palmer_2021}. These maps induce an isomorphism on $H_d$ in ranges

\[
d \le 
\begin{cases}
	\dfrac{n-1}{2}, & \text{if } p \le \frac{m-3}{2}, \\[6pt]
	n-1, & \text{if } p \leq \frac{m-3} 2 \text{ with } \mathbb{Z}\!\left[\tfrac12\right] \text{ coefficients,}  \\[10pt]
	\dfrac{n-3}{2}, & \text{if } p = \frac{m-2}{2}.
\end{cases}
\]

		\item If $M$ is closed, and $\ell$ a prime, for each $d$ there exists a $T$ which is a power of $\ell$,  and a range $r$ such that there are transfer maps $$H^d(\mrm{Sub}_{e_{n}}(P \times n, M), \bb F_\ell) \to H^d(\mrm{Sub}_{e_{n+T}}(P \times (n+T), M), \bb F_\ell)$$ that induce periodicity isomorphisms for every $n \geq r$. Moreover we can take 
			\[
		T \leq 
		\begin{cases}
			2d + 1 & \text{if } p \leq \frac{m - 3}{2}, \\
			4d + 5 & \text{if } p = \frac{m - 2}{2}.
		\end{cases}
		\]
		
		\[
		r \leq 
		\begin{cases}
			4d+3 & \text{if } p \leq \frac{m - 3}{2}, \\
			8d+11  & \text{if } p = \frac{m - 2}{2}.
		\end{cases}
		\]
		\item{If $M$ is closed, the cohomology with rational coefficients stabilizes, with slope \\ $d \leq \frac{n-2} 2$ if $p \leq \frac{m-3} 2$ and $d \leq \frac{n-6}{4}$ if $p = \frac{m-2} 2$.}
	\end{enumerate}
\end{thmalpha}

Note that we recover almost the same ranges as Palmer's in the case $p \leq \frac {m-3} 2$ and with $\bb Z$ coefficients. Our bound is worse by an additive $\frac 1 2$. The guiding philosophy of this paper, in the continuation of Church and Nagpal's papers \cite{Church_2012}, \cite{Nagpal_2015}, is that homological stability for a family of spaces (\( X_n \)) can sometimes be more easily established by identifying co\(\FI\)-spaces (\( Y_n \)) such that \( X_n \simeq (Y_n)_{hS_n} \), and then proving representation stability for the cohomology of \( Y_n \). We primarily explore this strategy in the case where \( X_n = \mathrm{Sub}_e(P \times n, M) \), the moduli space of submanifolds of a fixed isotopy class. In this context, a natural choice for \( Y_n \) is the space we will call \( \mathrm{PSub}_e(P \times n, M) \). This space parametrizes ordered copies of \( P \) embedded in \( M \).
\subsection*{Outline and strategy}
The proof proceeds in three steps. 
\begin{enumerate}
	\item{The input is the multiple disjunction lemma of Goodwillie and Klein for spaces of embeddings. This easily leads to a connectivity result for cubes of spaces of ordered embedded copies of submanifolds, with a suitable choice of isotopy class. We will apply Goodwillie's generalized Blakers-Massey theorem to deduce precise connectivity ranges for the cube of \emph{stable} homotopy types of spaces of submanifolds. This is not new but we include the precise computation for clarity.}
	\item{Using \cite[Theorem 2.2]{Gues_2025} we deduce that the cohomology of $\mrm{PSub}(P \times n, M)$ assembles into an $\FI$-module with finite presentation degree, with explicit linear bounds. }
	\item{Exploiting \cite{Nagpal_Snowden_2018}, we discuss two situations : 
	
	\begin{itemize}
		\item When the background manifold is open, the cohomology $H^*(\mrm{PSub}_e(P \times \bullet, M))$ has an $\FI\sharp$-module structure. Here $\FI \sharp$ is the category of finite sets and partially-defined injections. This additional structure allows to show that the family of quotient spaces $\mrm{Sub}(P\times \bullet, M) := \mrm {PSub}(P \times \bullet, M)/ S_\bullet$ is homologically stable over $\bb Z$, reproving a result of Palmer \cite{Palmer_2021}. We improve the bounds after inverting $2$ exploiting results obtained by Kupers and Miller \cite{Kupers_Miller_2015} and Cantero and Palmer \cite{Cantero_Palmer_2015}.
		\item When the background manifold is closed, Nagpal and Snowden's \cite{Nagpal_Snowden_2018} periodicity result yields that the cohomology of the symmetric group $S_n$ with coefficients in $H^d(\mrm{PSub}(P \times n, M), \bb F_\ell)$ is eventually periodic. A Postnikov tower argument, or equivalently inspecting the homotopy orbit spectral sequence 
		$$ E^2_{p,q} = H^p(S_n, H^q(\mrm{PSub}_{e_n}(P\times n, M), \bb F_\ell)) \implies H^{p+q}(\mrm{Sub}_{e_n}(P\times n, M), \bb F_\ell)$$
		 allows us to prove homological periodicity for $\mrm{Sub}_{e_\bullet}(P\times \bullet, M)$.
		\end{itemize}}
\end{enumerate}
In a last section, we will apply this line of reasoning to the group of symmetric diffeomorphisms of a closed manifold $M$, that is the subgroup of $\Diff(M)$ fixing $n$ distinguished points, generalizing results by Tillmann \cite{Tillmann_2016} for an open manifold.
\subsection{Related Works}
Several papers explored homological stability for spaces of embeddings or spaces of submanifolds. In addition to Palmer's paper \cite{Palmer_2021}, Kupers \cite{Kupers_2020} showed that homological stability for the space of embedded circles in $\bb R^3$. Note that we cannot deduce his result with our strategy because it is outside the dimensionality range $p \leq \frac{m-2}{2}$. The $\FI$-analogue of this result has been proved by Jennifer Wilson in \cite{Wilson_2012} for unlinked, ordered, embedded circles in $\bb R^3$. The paper of Cantero and Randal-Williams \cite{Cantero_Randal_2017} deals with embedded subsurfaces inside an open manifold, where the stabilization happens when increasing the genus.

\subsection*{Notations}
We will write $\ul n$ for the finite set $\{1,...,n\}$. We will often write $P \times n$ for the manifold $P \times \ul n$ by a slight abuse of notation. 
\subsection*{Ackowdlegements}
 I am very grateful to my advisor Geoffroy Horel for many discussions, support and proofreading during the writing of this paper. I thank Jeremy Miller, Arthur Soulié, Aurélien Djament and Martin Palmer for helpful conversations and suggestions. 
	\section{Connectivity estimates for moduli spaces of submanifolds}
The calculus of embeddings, developed by Weiss and Goodwillie \cite{Weiss_1999},\cite{Goodwillie_Weiss_1999} is a tool for analyzing contravariant functors from the category of manifolds (and embeddings) to spaces. The central idea is to approximate such functors via a \emph{Taylor tower}, analogous to the Taylor series of a function : 
	\[
	F \to \cdots \to T_2F \to T_1F \to T_0F
	\]
	where each \( T_kF \) is in a sense ``polynomial`` of degree \( k \). 
	
	In favorable cases, the tower converges to \( F \). In particular, the Taylor tower associated to the space of embeddings $\Emb(P,M)$ between two manfolds $P$ and $M$ converges if $\dim M - \hdim P \geq 3$ \cite{Goodwillie_Weiss_1999}, where $\hdim$ denotes the handle dimension . Moreover the connectivity of the map $\Emb(P,M) \to T_k \Emb(P,M)$ tends to infinity with $k$. The key property for this convergence result is the ``multiple disjunction`` \Cref{GKestimate} of Goodwillie and Klein for spaces of embeddings \cite{Goodwillie_Klein_2015}. 
\begin{thm}\cite[special case of Theorem B]{Goodwillie_Klein_2015}\label{GKestimate}
	Let $M$ be a smooth manifold of dimension $m$ and $P$ a closed smooth manifold of handle dimension $p$. Then for $n \geq 2$, the $n$-cube
	$$ \ul n \supseteq S \mapsto \mrm{Emb}(P \times S, M)$$
	is $3-m+n(m-p-2)$-cartesian. 
\end{thm}
Let us translate this statement into a connectivity result for the space of embedded submanifolds. 
Choose an embedding $P \xinj{e} \bb R^n $. By stacking copies of $e$ one next to each other in the first coordinate we obtain a preferred embedding of $n$ copies of $P$ : $P \times \ul n \inj \bb R^n$. Choose a small open ball $\mc U$ inside $M$ and a diffeomorphism $\bb R^n \xto{~} \mc U$. This gives a natural choice of embeddings of $n$ copies of  $P$ inside $M$ for each $n$. Call this map $e_n : P \times n \inj M$.  Define the co-$\FI$-space $\mrm{Emb}_{e}(P\times \bullet,M)$ by $\mrm {Emb}_{e}(P \times \ul n , M)$ being the connected component of $e_n$ inside $\mrm {Emb}(P\times \ul n , M)$. 
\begin{defn}
	Call $\mrm{PSub}(P \times n, M)$ the quotient of $\mrm{Emb}(P \times n, M)$ by the natural action of $\Diff(P)^n$. We call this space the ``pure`` moduli space of submanifolds, as it is the space of \emph{ordered} embedded copies of $P$ inside $M$. Call $\ol e_n$ the image of $e_n$ inside $\mrm{PSub}(P \times n, M)$ by the projection map $\Emb(P \times n, M) \to \mrm{PSub}(P \times n, M)$. Call $\mrm{PSub}_{\ol e_n}(P \times n, M)$ the connected component containing $\ol e_n$.   
\end{defn}
Our goal is now to understand the cartesianness of the $n$-cube that sends a set $S$ to $\mrm{PSub}_{\ol e_n}(P \times S, M)$. 

\begin{lemme}\label[lemme]{lemma:PSubCart}
	
	For $n \geq 2$, the $n$-cube $S \mapsto \mrm{PSub}(P \times S, M)$ has the same cartesianness as the corresponding cube of embeddings.
\end{lemme}
\begin{proof}
	There is a fiber sequence of $n$-cubes
	$$ \mathrm{Diff}(P)^S \to \mathrm{Emb}(P\times S , M) \to \mathrm{Emb}(P \times S, M)/\mathrm{Diff}(P)^S = \mrm{PSub}(P \times S, M)$$
	As the $n$-cube $S \mapsto \mathrm{Diff}(P)^S$ is (strongly) cartesian, the total fiber of $\mrm{PSub}(P \times n, M)$ is the same as the total fiber of $\mathrm{Emb}(P \times n, M)$.
\end{proof}
We can restrict this result to the connected component $\mrm{PSub}_{\ol e_n}(P\times n, M)$. 
\begin{lemme}
	The connectivity range of \Cref{lemma:PSubCart} holds for the 	cube
	$$S \mapsto \mrm{PSub}_{\ol e_n}(P \times S, M).$$
\end{lemme}
\begin{proof}
Restricting to the connected component of $\ol e_n$ does not change the total fiber at $\ol e_n$. This can for example be shown using an explicit model for the homotopy fiber. 
\end{proof}
We recall here Goodwillie's generalized Blakers-Massey theorem.
\begin{thm}\cite[Theorem 2.6]{Goodwillie_1992}\label{BMcocart}
	Let $\mathcal{X} = (T \mapsto X_T)$ be an $S$-cube with $|S| = n \geq 1$. Suppose that
	\begin{enumerate}
		\item for each $\vide \neq U \subset S$, the $U$-cube $\partial_{S-U} \mathcal{X}$ is $k_U$-cartesian, and
		\item for $U \subset V$, $k_U \leq k_V$.
	\end{enumerate}
	Then $\mathcal{X}$ is 
	\[
	(-1 + n + \min\{\Sigma_\alpha k_{T_\alpha}: \{T_\alpha\} \text{ is a partition of } S\})\text{-cocartesian}.
	\]
	In particular,  if $X$ is a strongly cartesian cube, then $X$ is $(-1+n + \sum_{i \in \ul n} k_i)$-cocartesian, writing $k_i$ for the connectivity of $X_{\ul n} \to X_{\ul n \setminus \{i\}}$.
\end{thm}
Now we compute the connectivity of the cube of stable homotopy types :
$$S \mapsto \Sigma^\infty \mrm {PSub}_{\ol e_n}(P \times S, M))$$
\begin{prop}
	Suppose $m \geq 3$. Then
	\begin{enumerate}
		\item If $p \leq \frac{m-3}{2}$ then the $n$-cube of stable homotopy types is $(n-1)$-cartesian. 
		\item If $\frac{m-3} 2 < p \leq \frac{m-1}{2}$ it is only $\frac{n}{2}(m-2p-1)$-cartesian. In particular the only interesting case is when $2p  = m-2$ so that the cube is $\frac{n}{2}$-cartesian. 
	\end{enumerate}
\end{prop}
\begin{proof}
Note that in our case the connectivities $k_U$ only depend on $|U|$. We will write $k_m$ for the connectivity of subcubes of size $m$. Moreover, to apply the Blakers-Massey theorem we need choices of $k_U$ that are increasing. For each singleton $\{m\}$, the connectivity $k_{\{m\}}$ is equal to $1$ since the map $\mrm{PSub}_{\ol e_1}(P, M) \to \star$ is $1$-connected (the fiber is $0$-connected). Hence we need $k_2 \geq 1$. In our case $k_2 = m-2p-1$ so we need $p \geq \frac{m-2}{2}$. From now on we suppose this condition is satisfied. The nondecreasing of the connectivities $k_i$ is obvious to check in the other cases. By the Generalized Blakers-Massey theorem, we get that the $n$-cube of disconnected submanifolds is 
$$-1 + n + \min \left\{\sum_{T_\alpha}k_{T_\alpha}, \alpha \text{ partition of } \ul n \right\}\text{- cocartesian.}$$

Fix a partition $(T_\alpha)$ of $\ul n$ and call $s$ the number of singletons in the partitions, $t = |\alpha| - s$ the number of non-singletons. 

For each singleton $\{m\}$, the connectivity $k_{m}$ is equal to $1$. For a partition of size $r \geq 2$ the connectivity is determined by Goodwillie and Klein's estimates and we have $k_{T_\alpha} = 3 - m + r(m - p - 2)$.  

We deduce that for a given partition, the sum of connectivities is 
$$ s + t(3 - m) + (n - s)(m - p - 2) $$
$$ = n(m - p - 2) + s(3 - m + p) + t(3 - m) $$

As $3 - m \geq 0$, the minimum value happens for the biggest $t$. By basic counting, $s + 2t \geq n$ so we have at worst $t = \frac{n - s}{2}$. This yields 
$$\min{\sum_{T_\alpha} k_{T_\alpha}} \geq s\left(\frac{3}{2} - \frac{m}{2} + p\right) + n\left(\frac{m}{2} - p - \frac{1}{2}\right)\footnote{Note that this estimation is suboptimal. With a more careful inspection involving the combinatorics of the partitions, Cihan Bahran suggested that we could improve slightly the bounds in this theorem.}$$

If $p >\frac{m - 3}{2}$, then the minimum of this sum is attained for $s = 0$ and it is $n\left(\frac{m}{2} - p - \frac{1}{2} \right)$. This yields that the original $n$-cube is $-1 + n\left(\frac{m}{2} - p + \frac{1}{2}\right)$-cocartesian and thus that the cube of stable homotopy types is $-1 + n\left(\frac{m}{2} - p - \frac{1}{2}\right)$-cartesian. In the case when $p = \frac{m-2} 2$ the cube is $(\frac n 2 - 1)$-cartesian.\\

If $p \leq \frac{m - 3}{2}$, the minimum is attained for $s = n$ and the original $n$-cube is $(-1 + 2n)$-cocartesian. Hence, the cube of stable homotopy types is $(n - 1)$-cartesian.

\begin{rmq}
	Goodwillie and Klein prove also a weaker and easier disjunction lemma, see the proposition $B.7$ of \cite{Goodwillie_Klein_2008}. This weaker lemma relies on much less sophisticated tools, namely transversality and the Blakers–Massey theorem. The bounds obtained via this method are frustratingly close to those derived from the much more involved machinery of Goodwillie and Klein: the two differences is that the hard one additionally covers the case $p =\frac{m - 2}{2}$, and depends on the handle dimension of $P$ instead of the genuine dimension. This is close to Palmer's result, where the condition was $p \leq \frac{m - 3}{2}$. From another perspective, this shows that many cases (in particular all those covered in \cite{Palmer_2021}) we prove can be recovered using this far more elementary input. The loss happens when passing from unstable to stable homotopy types. To extract stable connectivity bounds, one must feed each result (the strong and the weak one) into the Blakers--Massey machinery. In doing so, the two distinct ranges yield almost identical dimensionality conditions for stable homotopy types \(p \leq \frac{m-2}{2}\) (resp.\ \(\frac{m-3}{2}\)). 
	
\end{rmq}
\end{proof}	

\section{The representation stability step}
Call $\FI$ (resp. $\FB$, $\FI\sharp$) the category of finite sets and injections (resp. bijections, partially-defined injections). We recall the definition of $\FI$-homology. For $\mc A$ an abelian category with enough projectives, define a functor $\bb H_0 : \mbf{Fun}(\FI, \mc A) \to  \mbf{Fun}(\FB, \mc A)$, where $\FB$ is the category of finite sets and bijections, through the formula

$$ \bb H_0 V(T)  = \coker(\bigoplus_{S \subsetneq T} V(S) \to V(T)).$$
for $V$ an $\FI$-module. The $p$-th $\FI$-homology functor $\bb H_p$ is defined to be the $p$-th left derived functor of $ \bb H_0$.

This definition can also be extended to general $\FI$-objects in an $\infty$-category $\mc C$, as the left derived functor of the extension-by-zero functor $ \mbf{Fun}(\FB, \mc C) \to  \mbf{Fun}(\FI, \mc C)$, see \cite{Gues_2025}. For an $\FI$-spectrum $X$, we get an $\FB$-spectrum $\bb HX$.
For $V$ an $\FI$-module, call $$t_k V = \max  \{n \text{ such that } \bb H_k(V)(\ul n) \neq 0\}.$$ 
We can also generalize the definition of $t_k$ for $\FI$-spectra with $$\mbf t_k X = \max  \{n \text{ such that } \pi_k \bb H X(\ul n) \neq 0\}.$$ The bold font is to emphasize that we apply this construction to an object in a non-abelian case.  
An $\FI$-module $V$ is said to be generated in finite degree (resp. presented in finite degree) if $t_0 V < \infty$ (resp. $t_1 V < \infty$).

In this section we show how the connectivity ranges of the previous section translate into representation stability theorems for spaces of submanifolds. We deduce finite generation/presentation degree for the integral cohomology groups $H^d(\mrm{PSub}_{\ol e_\bullet}(P \times \bullet, M))$. We will use the following consequence of Theorem 2.2 from \cite{Gues_2025}. 
\begin{thm}\cite[special case of Theorem 2.2]{Gues_2025}\label{propergoingdown}
	Let $X$ be a co$\FI$-spectrum.
	Suppose $\mbf t_p X \leq f(p)$ with $f$ nonincreasing. Then 
		\begin{enumerate}
			\item{$t_0 \pi_{p} X \leq \max(-1, 2f(p))$}
			\item{$t_1 \pi_{p} X \leq \max(0, 2f(p)+1)$}
		\end{enumerate} 
\end{thm}
When $X$ is an co$\FI$-spectrum such that all its homotopy groups have an $\FI \sharp$-structure (in particular, when $X$ is itself $\FI \sharp$), we can do better with a simpler proof :
\begin{lemme}\label[lemma]{fisharpgoingdown}
	Suppose the homotopy groups $\pi_p X$ have an $\FI \sharp$-module structure. Then : 

$$t_0 \pi_{p} X = \mbf t_p X$$
\end{lemme}
\begin{proof}
	The proof uses a simple spectral sequence argument. The $\FI$-hyperhomology spectral sequence has the following $E_2$ page :
	$$ E_2^{p,q}= \bb H_p \pi_q X\implies \bb H_{p+q} X$$
	But by \cite{Church_Ellenberg_2017}, $\FI \sharp$-modules are $\bb H_0$-acyclic. Therefore $E_2^{p,q} = 0$ as soon as $p \neq 0$. The spectral sequence degenerates at $E_2$ and we deduce $\bb H_p X = \bb H_0 \pi_p X$ hence the result. 
	\end{proof}
We will apply these two results to the cochain complex of the spaces we study, seen as a spectrum concentrated in non-positive degrees. Recall that a (co)chain complex $C$ can be seen as an $H\bb Z$-module and  its homotopy groups are precisely the classical (co)homology groups of $C$.  Note that this approach allows to deduce representation stability for other cohomology theories, for example if we apply it to the (Anderson, or Brown-Comenetz duals of the) suspension spectrum of the embedding spaces we get representation stability for the (dual) stable homotopy groups. \\
Combining this result with the connectivity estimates of section $1$, we get the following bounds on the presentation degrees. 
\begin{prop}\label[prop]{FI_PSub}
	For each integer $d$ the $\FI$-module
	$$ S \mapsto H^d(\mathrm{PSub}_{\ol e_S}(P \times S, M)) $$ is presented in finite degrees. More explicitly we have 
	\begin{enumerate}
		\item{If $M$ is an open manifold, then the cohomology groups are $\FI \sharp$-modules and
\[
t_0 H^d(\mathrm{PSub}_{\overline{e}_\bullet}(P \times \bullet, M)) \leq 
\begin{cases}
	d + 1 & \text{if } p \leq \frac{m - 3}{2}, \\
	2d + 3 & \text{if } p = \frac{m - 2}{2}.
\end{cases}
\]  }
		\item{If $M$ is a closed manifold, then the cohomology groups are only $\FI$-modules and
		\begin{enumerate}
			\item 
			\[
			t_0 H^d(\mathrm{PSub}_{\overline{e}_\bullet}(P \times \bullet, M)) \leq 
			\begin{cases}
				2d + 2 & \text{if } p \leq \frac{m - 3}{2}, \\
				4d + 6 & \text{if } p = \frac{m - 2}{2}.
			\end{cases}
			\] 
			\item 
			\[
			t_1 H^d(\mathrm{PSub}_{\overline{e}_\bullet}(P \times \bullet, M)) \leq 
			\begin{cases}
				2d + 3 & \text{if } p \leq \frac{m - 3}{2}, \\
				4d + 7 & \text{if } p = \frac{m - 2}{2}.
			\end{cases}
			\] 
			\end{enumerate} }
	\end{enumerate}

\end{prop}

\section{From representation stability to homological stability}
In this section we use the periodicity result by Nagpal and Snowden \cite{Nagpal_Snowden_2018} on the cohomology of the symmetric group with coefficients in a finitely generated $\FI$-module. Combined with a Postnikov tower argument, we show that if $X$ is a co$\FI$-space for which the cohomology is finitely generated as an $\FI$-module, then the sequence of spaces $(X_n)_{h S_n}$ satisfies 
homological periodicity over $\bb F_\ell$, and stability over $\bb Q$. Furthermore, if we start with a space whose cohomology has a $\FI \sharp$-structure, this can be improved to homological \emph{stability} over $\bb Z$. The arguments are very similar to the application to unordered configuration spaces in \cite{Nagpal_2015}.
We will need to show a generalization of Nagpal and Snowden's results that uses only the hypothesis that the $\FI$-module has finite presentation degrees, without being necessarily finitely generated.
\subsection{Homological stability when the target manifold has boundary}
In this section we show homological stability for spaces of embedded copies of $P$ inside an manifold with boundary $M$. We recover results by Palmer \cite{Palmer_2021} with another strategy. 
\begin{prop}
	When the target manifold is open, the spaces of submanifolds $$\mrm{PSub}_{\ol {e}_n}(P \times \bullet, M)$$ can be given a structure of a homotopy $\FI\sharp$-space.
\end{prop}
\begin{proof}
	The proof of \cite[ Proposition 6.4.2]{Church_Ellenberg_2015} adapts. We just need to replace every point in the configuration by a disjoint embedded copy of $P$.
\end{proof}

Note that we do not claim that we can lift this structure to a \emph{strict} or \emph{homotopy-coherent} $\FI \sharp$-structure.
\begin{lemme}[{Homological stability for symmetric groups over $\mathbb{Z}\left[\frac{1}{2}\right]$}]\label[lemme]{nakaokaawayfrom2}
		The stabilization map $H_k(S_n, \bb Z) \to H_k(S_{n+1}, \bb Z)$ is an isomorphism if $k \leq \frac n 2$.	With coefficients in $\bb Z[\frac 1 2]$, the range improves to $k \leq n$.
\end{lemme}

\begin{proof}
	This can be shown for example from the analogous result for unordered configuration spaces of open manifolds, proven independently by Cantero-Palmer and Kupers-Miller \cite{Cantero_Palmer_2015}, \cite{Kupers_Miller_2015}. Taking $M= \bb R^d$ and making $d$ go to infinity gives the result for the symmetric groups because the unordered configuration space in $\bb R^\infty$ is a classifying space of the symmetric group, and the map $\mc{U}\conf_n(\bb R^d) \to \mc{U}\conf_n(\bb R^{\infty})$ is highly connected. This can also be deduced from Cohen's calculations \cite{Cohen_1976} of the cohomology of configuration spaces of $\bb R^n$ over $\bb F_\ell$, see \cite[Proposition 3.3]{Kupers_Miller_2015}.
\end{proof}
\begin{prop}\label[prop]{FISharpCoinvariants}
	If $V$ is an $\FI\sharp$-module freely generated by an $\FB$-module $W$ concentrated in cardinality $k$ then the natural map $V_n \to V_{n+1}$ induces the map
	$$ W_{h S_k} \otimes C_*(B S_{n-k}) \xto{id \otimes p}  W_{h S_k} \otimes C_*(B S_{n-k+1})$$
	on the homotopy coinvariants for the symmetric groups, where $p$ is induced by the natural injection between symmetric groups. The tensor product $\otimes$ is derived. 
\end{prop} 
\begin{proof}
$V$ is the free $\FI$-module on $W$, so we have $V(n) = \bigoplus_{S \subset \ul n, |S| = k} W_S$. In other terms $V_n = \mrm{Ind}_{S_k \times S_{n-k}}^{S_n} W$. The map $V_n  \to V_{n+1}$ is the inclusion map. By construction, for every inclusion $H \subset G$ of groups and $X$ a chain complex, $(\mrm {Ind}^G_H X)_{hG} \simeq X_{hH}$ so we reduce to the induced map
$$ W_{h(S_k \times S_{n-k})} \to W_{h(S_k \times S_{n+1-k})}.$$
As $W_{h(S_k \times S_{n-k})} = (W_{hS_k})_ {hS_{n-k}}$ it suffices to look at the map 
$$ (W_{hS_k})_{hS_{n-k}} \to (W_{hS_k})_{hS_{n+1-k}}$$
But as the action of $S_{n-k}$ and $S_{n+1-k}$ is trivial on $W$, the coinvariants are simply tensoring with the homotopy quotient of the point, or in other words, tensoring with $C_*(BS_{n-k})$. We get the map
	$$ W_{h S_k} \otimes C_*(B S_{n-k}) \xto{id \otimes p}  W_{h S_k} \otimes C_*(B S_{n-k+1})$$
\end{proof}
\begin{coroll}\label{FIsharpmodulestability}
	Let $V$ be an $\FI \sharp $-module generated in cardinality $\leq k$. Then the map between complexes
	$$ (V_n)_{h  S_n} \to (V_{n+1})_{h  S_{n+1}}$$
	is $\frac{n-k} 2$-connected. Moreover, if we invert $2$, the map is $(n-k)$-connected. 
\end{coroll}
\begin{proof} By \cite[Theorem 4.1.5]{Church_Ellenberg_2015} any $\FI\sharp$-module is of the form $M(W)$ for some $\FB$-module $W$.
	 As every $\FB$-objet in the derived category splits as a sum of $\FB$-modules concentrated in one cardinality, it is sufficient to show the result for a $\FI\sharp$-module $V$ generated freely by an $\FB$-module concentrated in a cardinality $k$.
	In this case, the map $V_n  \to V_{n+1}$ induces the map described in \Cref{FISharpCoinvariants}. This map is  $\frac{n-k} 2$-connected by \Cref{nakaokaawayfrom2}. Moreover, if we invert $2$, the stability ranges are twice as good.
\end{proof}

\begin{coroll}
	The natural map $C_*(\mrm{Sub}_{\ol e_{n}}(P \times n, M)) \to  C_*(\mrm{Sub}_{\ol e_{n+1}}(P \times (n+1), M)) $
	\begin{enumerate}
		\item $\frac {n-1} 2$-connected over $\bb Z$ and $(n-1)$-connected over $\bb Z[\frac 1 2]$ when $p \leq \frac {m-3} 2$,
		\item $\frac {n-3} 2$-connected over $\bb Z$ when $p = \frac {m-2} 2$ with no significant improvement over $\bb Z[\frac 1 2]$. 
	\end{enumerate}
\end{coroll}
\begin{proof}
	For readability, we write $\mrm{PSub}(k)$ for $\mrm{PSub}_{\ol e_{k}}(P \times k, M)$ and similarly $\mrm{Sub}(k)$. 
Consider the map adding a copy of $P$ on the boundary : $$C_*(\mrm{PSub}(n)) \to  C_*(\mrm{PSub}(n+1)) $$ which is induced by the homotopy $\FI\sharp$-structure.  This map is compatible with the action of $S_{n-1}$ (not only up to homotopy). This induces the quotient map
$$ C_* (\mrm{Sub}_{n-1}) \to C_* (\mrm{Sub}_n)$$
of which we want to control the connectivity. Consider the tower of truncations $C_* \mrm{PSub}_n \to \cdots \to \tau_{\leq *} C_* \mrm{PSub}_n \to \cdots$. There is a fiber sequence  
 $$ H_d \mrm{PSub}_n [d] \to  \tau_{\leq d} C_* \mrm{PSub}_n \to \tau_{\leq d-1}C_* PSub_n$$ which inherits an action of $S_n$. Taking coinvariants we get
 $$ (H_d \mrm{PSub}_n [d])_{hS_n} \to  (\tau_{\leq d} C_* \mrm{PSub}_n)_{hS_n} \to (\tau_{\leq d-1}C_* \mrm{PSub}_n)_{hS_n} $$
By functoriality we get a commutative diagram of chain complexes whose rows are fiber sequences :
\[\begin{tikzcd}
	{(H_d \mrm{PSub}_{n-1})_{hS_{n-1}}[d]} & {(\tau_{\leq d} C_* \mrm{PSub}_{n-1})_{hS_{n-1}}} & {(\tau_{\leq d-1}C_* \mrm{PSub}_{n-1})_{hS_{n-1}}} \\
	{(H_d \mrm{PSub}_{n})_{hS_{n}}[d]} & {(\tau_{\leq d} C_* \mrm{PSub}_{n})_{hS_{n}}} & {(\tau_{\leq d-1}C_* \mrm{PSub}_{n})_{hS_{n}}}
	\arrow[from=1-1, to=1-2]
	\arrow[from=1-1, to=2-1]
	\arrow[from=1-2, to=1-3]
	\arrow[from=1-2, to=2-2]
	\arrow[from=1-3, to=2-3]
	\arrow[from=2-1, to=2-2]
	\arrow[from=2-2, to=2-3]
\end{tikzcd}\]
Consider first the case $p \leq \frac{m-3} 2$. Combining \cref{FI_PSub} and \cref{FIsharpmodulestability}, the vertical leftmost map is $\frac{n+d-1}{2}$-connected. By an easy induction argument, we deduce that for all $d$ the map 
$$(\tau_{\leq d} C_*(\mrm{PSub}_{n}))_{hS_{n}} \to (\tau_{\leq d} C_*(\mrm{PSub}_{n+1}))_{hS_{n+1}}$$
 is $\frac{n}{2}$-connected. The map $C_* \mrm{Sub}_n = C_*(\mrm {PSub}_n)_{hS_n} \to (\tau_{\leq n}(C_*(\mrm{PSub}_n)))_{hS_n}$ is $n+1$-connected (it has an $n$-connected fiber), and the following square is commutative
 \[\begin{tikzcd}
 	{C_*(\mrm{Sub}_n)} & {(\tau_{\leq n}(C_*(\mrm{PSub}_n)))_{hS_n}} \\
 	{C_* (\mrm{Sub}_{n+1})} & {(\tau_{\leq n+1}(C_* (\mrm{PSub}_{n+1})))_{hS_{n+1}}}
 	\arrow[from=1-1, to=1-2]
 	\arrow[from=1-1, to=2-1]
 	\arrow[from=1-2, to=2-2]
 	\arrow[from=2-1, to=2-2]
 \end{tikzcd}\]
 This implies that the map $C_* (\mrm{Sub}_{n}) \to C_* (\mrm{Sub}_{n+1})$ is $\frac{n-1}{2}$-connected. When working over $\bb Z[\frac 1 2]$, the same reasoning shows that the map is $(n-1)$-connected. Now let us deal with the case $p = \frac{m-2} 2$. This time the map $ (H_d \mrm{PSub}_{n})_{hS_{n}}[d] \to 	(H_d \mrm{PSub}_{n+1})_{hS_{n+1}}[d]$ is $\frac{n-(2d+3)} 2 + d$ = $\frac{n-3}{2}$-connected, and we get that $C_* (\mrm{Sub}_{n}) \to C_* (\mrm{Sub}_{n+1})$ is $\frac{n-3} 2$-connected. 

 Note that in the last case, with $\bb Z[\frac 1 2]$ coefficients, the map $ (H_d \mrm{PSub}_{n})_{hS_{n}}[d] \to 	(H_d \mrm{PSub}_{n+1})_{hS_{n+1}}[d]$ is $n-(2d+3) + d$ = $(n-d-3)$-connected. By the same induction argument, this implies that for any $d$ the map $(\tau_{\leq d} C_* (\mrm{PSub}_{n}))_{hS_{n}} \to (\tau_{\leq d} C_* (\mrm{PSub}_{n+1}))_{hS_{n+1}}$ is $(n-d-3)$-connected. 
Fix an integer $d$. The map $C_* Sub_n = C_* (\mrm{PSub}_n)_{hS_n} \to (\tau_{\leq d}(C_* (\mrm{PSub}_n)))_{hS_n}$ is $d+1$-connected. Hence the map $C_* (\mrm{Sub}_{n}) \to C_* (\mrm{Sub}_{n+1})$ is $\min(n-d-3 , d+1)$-connected. Taking the maximum for $d$ ranging over $\bb N$ we get a connectivity of $\frac{n-2} 2$, which is only $+ \frac 1 2$ compared to the integral ranges.  

\end{proof}
 This proves the first part of \Cref{mainresult}. 
\subsection{Homological periodicity with a closed target manifold}
In this part our main tool will be Nagpal and Snowden's periodicity result. We will recall their main theorem and generalize it to $\FI$-modules that have finite presentation degree but are not necessarily finitely generated. 
\begin{defn}
	The free divided power algebra on one generator $\Gamma(t)$ is the algebra freely spanned by elements $t^{[k]}$ (thought of as $\frac{t^k}{k!}$) with the relations 
	$$t^{[m]} t^{[n]} = \binom{n+m}{n} t^{[m+n]}$$
\end{defn} 
\begin{prop}(\cite[Proposition 4.4]{Nagpal_Snowden_2018}).
	Let $V$ be an $\FI$-module over a ring $\mbf k$. Then
	$$\bigoplus_n H^d(S_n, V_n) $$
	has a structure of a module over the divided power algebra $\Gamma(t)$. The map induced by multiplication by $t^{[k]}$ is the composition
	$$ H^d(S_n, V_n) \xto{e} H^d(S_n \times S_k, V_n) \xto{i} H^d(S_n \times S_k, V_{n+k}) \xto{c} H^d(S_{n+k}, V_{n+k})$$
	Where $e$ is the pullback along the projection $S_n \times S_k \to S_n$ with $S_k$ acting trivially on $V_n$,  $i$ is induced by any map $V_n \to V_{n+k}$ and $c$ is the corestriction along the inclusion $S_n \times S_k \subset S_{n+k}$
\end{prop}
The main theorem of \cite{Nagpal_Snowden_2018} is that if $V$ is a finitely generated $\FI$-module over $\bb F_\ell$ then the map induced by $t^{[k]}$ is eventually an isomorphism for $n$ big and $k$ a sufficiently big power of $p$. Quantitatively we have
\begin{thm}\cite[Theorem 4.12, 4.15]{Nagpal_Snowden_2018} \label{NSPeriodic} 
	Let $V$ be a finitely generated $\FI$-module over $\bb F_\ell$. 
	For each $d \geq 0$, there exists a period $q$ and a range $r$ such that for all $n \geq r$,
	the multiplication by $t^{[q]} : H^d(S_n, V_n) \to H^d(S_{n+q}, V_{n+q}) $ is an isomorphism. Moreover, we can take
	
	$$r = \max\left(t_0(V) + t_1(V), 2t_0(V), 2d + t_0(V) \right)$$
	and $q$ to be any power of $p$ strictly bigger than $t_0(V)$. 
	In homotopical terms, this amounts to say that the transfer map between complexes
	$$ (V_n)^{hS_n} \to (V_{n+q})^{hS_{n+q}}$$
	is $\frac{n-t_0}{2}$-connected as soon as $n \geq \max(t_0+t_1, 2t_0)$ (we adopt a cohomological grading here). 
\end{thm}
Now we want to generalize Nagpal and Snowden's periodicity result for $\FI$-modules with finite presentation degree, which are not necessarily finitely generated. We begin with some remarks.
\begin{rmq}
	If $V \to W$ is a morphism of $\FI$-modules, the comparison maps $H^d(\Sigma_n, V_n) \to H^d(\Sigma_n, W_n)$ is a morphism of $\Gamma(t)$-modules. 
\end{rmq}
Recall  (see \cite[Section 4.1]{Church_Ellenberg_2015}) that $g$ and $p$ are almost equal to the homological invariants $t_0$ and $t_1$ but not exactly. We have in fact $t_0 = g$ and $t_1 \leq p \leq \max(t_0,t_1)$
\begin{lemme}
	Let $V$ be an $\FI$-module with generation and presentation degrees $g$ and $p$. Then $V$ is a filtered colimit of finitely generated $\FI$-modules $V_i$ with same generation and presentation degrees.
\end{lemme}
\begin{proof}
	We can write $V$ as a cokernel
	$$ V = \coker \left(\bigoplus_{i \in I} M(m_i) \to \bigoplus_{j \in J} M(n_j)\right)$$
	With all the integers $n_j$ smaller than $g$ and $m_i$ smaller than $p$. Then filter $I$ and $J$ by their finite subsets. We have
	$$ V = \colim_{\mc P_f(I \times J)}  \coker \left(\bigoplus_{i\in I'} M(m_i) \to \bigoplus_{j \in J'} M(m_j) \right)$$ with the obvious maps. The indexing category $\mc P_f(I \times J)$ is the category of finite subsets of $I \times J$, which is filtered.  Each cokernel has generation degree $\leq g$ and presentation degree $\leq p$ by construction.  
\end{proof}
By functoriality of the periodicity maps, we deduce that the periodicity isomorphisms can be extended to $\FI$-modules with finite presentation degrees. 
\begin{coroll}
	\Cref{NSPeriodic} extends to $\FI$-modules such that $t_0 , t_1 \leq \infty$, with the same bounds on the periodicity onset and period.  
\end{coroll}
\begin{proof}
	Cohomology of groups commutes with filtered colimits. The comparison map
	$$ H^p(S_n, V_n) \to H^p(S_{n+k}, V_{n+k} )$$ is functorial and hence the filtered colimit of the maps $H^p(S_n, V^{I,J}_n) \to H^p(S_n, V^{I,J}_n)$ which are isomorphisms for $n$ sufficiently big, independently of $I$ and $J$. We obtain the same periodicity range and period than for the finitely presented $\FI$-modules $V^{I,J}$ and hence the same bounds as in \Cref{NSPeriodic}.
\end{proof}
When the target manifold $M$ is closed, there is no natural stabilization map anymore. Following the notations of (part 1), a small ball $B$ inside $M$ and a preferred embedding $e_n : P \times n \to B$. The stability results will concern the sequence of spaces $\mrm{PSub}_{\ol e_n}(P \times n, M)$ of connected components of the moduli space of submanifolds. We will use a spectral sequence argument and the results of Nagpal and Snowden \cite{Nagpal_Snowden_2018}.
\begin{lemme}
	For $X$ a co$\FI$-space and $d$ a cohomological degree, the cohomology $$\bigoplus_{n \in \bb N} H^d((X_n)_{hS_n})$$ is a module over the divided power algebra $\Gamma(t)$. 
\end{lemme}
\begin{proof}
	This follows exactly the same proof as Proposition 4.4 \cite{Nagpal_Snowden_2018}. 
\end{proof}

We can now apply the periodicity result of Nagpal and Snowden \cite{Nagpal_Snowden_2018}.  We deduce
\begin{prop}
	Let $X$ be a co$\FI$-space whose $\bb F_\ell$-cohomology $H^d$ satisfies $t_0 \leq ad+b$ and $t_1 \leq ad+b+1$ for all $d$. Then the cohomology groups $H^d((X_n)_{hS_n})$ are eventually periodic with period $T$ the smallest power of $p$ greater than $ad+b$. The periodicity onset is $2ad+2b+1$ if $a >2$ and $\max\{2d+b+2, 2ad+2b+1\}$ if $a \geq2$. 
\end{prop}
\begin{proof}
	All complexes and cohomology groups are with $\bb F_\ell$ coefficients. First note that the transfer map $$ H^d(S_n, V_n) \to H^d(S_{n+T},V_{n+T})$$ can be lifted at the cochain level :
	$$ C^*(X_n)^{hS_n} \to C^*(X_n)^{h(S_n \times S_T)} \to C^*(X_{n+T})^{h(S_n \times S_T)} \xto{\mrm{corestriction}} C^*(X_{n+T})^{h(S_{n+T})}$$ 
	For $X$ an arbitrary co$\FI$-space, and the induced maps on cohomology are the transfer maps. Now apply the tower of truncations $\tau_{\geq d}$ to the diagram of cochain complexes before taking the $S_n$-invariants. Taking all the fixed points commutes with fiber sequences and we get a series of diagrams
\[\begin{tikzcd}
	{(\tau_{\leq d-1}C^* \mrm{PSub}_{n})^{hS_{n-1}}} & {(\tau_{\leq d} C^* \mrm{PSub}_{n})^{hS_{n}}} & {(H^d \mrm{PSub}_{n})[d]^{hS_{n}}} \\
	{(\tau_{\leq d-1}C^* \mrm{PSub}_{n})^{h(S_{n}\times S_T)}} & {(\tau_{\leq d} C^* \mrm{PSub}_{n})^{h(S_{n}\times S_T)}} & {(H^d \mrm{PSub}_{n})[d]^{h(S_{n}\times S_T)}} \\
	{(\tau_{\leq d-1}C^* \mrm{PSub}_{n+T})^{h(S_{n}\times S_T)}} & {(\tau_{\leq d} C^* \mrm{PSub}_{n+T})^{h(S_{n}\times S_T)}} & {(H^d \mrm{PSub}_{n})[d]^{h(S_{n+T}\times S_T)}} \\
	{(\tau_{\leq d-1}C^* \mrm{PSub}_{n+T})^{hS_{n+T}}} & {(\tau_{\leq d} C^* \mrm{PSub}_{n+T})^{hS_{n+T}}} & {(H^d \mrm{PSub}_{n+T})[d]^{hS_{n+T}}}
	\arrow[from=1-1, to=1-2]
	\arrow[from=1-1, to=2-1]
	\arrow[from=1-2, to=1-3]
	\arrow[from=1-2, to=2-2]
	\arrow[from=1-3, to=2-3]
	\arrow[from=2-1, to=2-2]
	\arrow[from=2-1, to=3-1]
	\arrow[from=2-2, to=2-3]
	\arrow[from=2-2, to=3-2]
	\arrow[from=2-3, to=3-3]
	\arrow[from=3-1, to=3-2]
	\arrow["tr", from=3-1, to=4-1]
	\arrow[from=3-2, to=3-3]
	\arrow["tr", from=3-2, to=4-2]
	\arrow["tr", from=3-3, to=4-3]
	\arrow[from=4-1, to=4-2]
	\arrow[from=4-2, to=4-3]
\end{tikzcd}\]
	whose rows are exact. The last vertical arrows are the transfer maps (or corestriction).  We then conclude by an induction argument similar to the proof of \Cref{FI_PSub}. Fix an integer $d$. Let 	
	$T = \max_{k \leq d}(t_0(H^k \mrm{PSub}_*)) \leq ad+b$. Each associated graded in the tower is a shifted cohomology group. Let us compute the connectivity of this map. 
	By Nagpal and Snowden's \Cref{NSPeriodic},
	
	$$	(H^d \mrm{PSub}_{n})^{hS_{n}}\to (H^d \mrm{PSub}_{n+T})^{hS_{n+T}}$$
	induces an isomorphism in homological degree $k$ if $n \geq \max(t_0+t_1, 2t_0, 2k+t_0)$, where $t_0$ and $t_1$ are those of $H^d \mrm{PSub}_*$. By our assumptions this simplifies to $n \geq \max(2ad+2b+1, ad+b+2k )$. Hence, as soon as $n \geq 2ad+2b+1$, we have an isomorphism in degree $k$ if $n \geq ad+b+2k$. Reformulating this in terms of connectivity, we get that the map is $\frac{n-ad-b} 2$-connected as soon as $n \geq 2ad+2b+1$. 
	Suspending $d$ times, the vertical right most map in the previous diagram is $\frac{n-(2a-2)d-b} 2$-connected as soon as $n \geq 2ad+2b+1$. By an easy induction on $d$, the central vertical map between the $d$-truncations is $\min_{k \leq d}  \frac{n-(a-2)k-b}{2} $-connected. The comparison map $$ (\tau_{\leq d}C^* \mrm{PSub}_n)^{hS_n} \to (C^* \mrm{PSub}_n)^{hS_n}$$ is $d$-connected, hence the connectivity of  $(C^* \mrm{PSub}_n)^{hS_n} \to (C^* \mrm{PSub}_{n+T})^{hS_{n+T}}$ has connectivity at least $\max_d \min(d, \min_{k \leq d} \frac{n-(a-2)k-b}{2})$. Let us compute more explicitly. Suppose first that $a \leq 2$. In this case $\max_d \min(d,\min_{k \leq d} \frac{n -(a-2)k-b}{2}) = \frac{n-b} 2$. Note that this holds only for $n \geq 2ad+2b+1$. We deduce that the induced map in homological degree $d$ is an isomorphism for $n \geq \max\{2(d+1)+b+1, 2ad+2b+1\}$. The $+1$ added comes from the fact that an $n$-connected map is an isomorphism in homology up to degree $n-1$.  Hence the periodicity onset is smaller than $\max\{2(d+1)+b, 2ad+2b+1\}$ for the $d$-th cohomology group. \\
	Now suppose $a > 2$. we have $\min_{k \leq d} \frac{n -(a-2)k-b}{2} = \frac{n-(a-2)d-b} 2$. The maximum value of $\min(d, \frac{n-(a-2)d-b} 2)$ is attained for $d  = \frac{n-(a-2)d-b} 2$, which gives $d = \frac{n-b} a$. Thus the map is $\frac{n-b}{a}$-connected. Consequently the periodicity onset $r$ satisfies $r \leq \max(a(d+1)+b, 2ad+2b+1) = 2ad+2b+1$. 
\end{proof}
Plugging the ranges of \Cref{FI_PSub} into the previous proposition we obtain the following result.
\begin{coroll}
	Over $\bb F_\ell$, the groups $H^d(\mrm{Sub}_{\ol e_\bullet}(P\times \bullet,M), \bb F_\ell)$ are eventually periodic with period $T$ a power of $p$ and an onset $r$ satisfying : 
	\[
	T \leq 
	\begin{cases}
		2d + 1 & \text{if } p \leq \frac{m - 3}{2}, \\
		4d + 5 & \text{if } p = \frac{m - 2}{2}.
	\end{cases}
	\]
	
	\[
	r \leq 
	\begin{cases}
		4d+3 & \text{if } p \leq \frac{m - 3}{2}, \\
		8d+11  & \text{if } p = \frac{m - 2}{2}.
	\end{cases}
	\]
\end{coroll}
This proves the second part of \Cref{mainresult}.
\subsection{The rational case.}
With $\bb Q$ coefficients, we simply have $$H^d(\mrm{Sub}_{\ol e_n} (P \times n, M)) = H^d(\mrm{PSub}_{\ol e_n}( P \times n, M))^{S_n} = H^d(\mrm{PSub}_{\ol e_n}(P \times n, M)_{S_n})$$. We can therefore deduce rational homological stability thanks to the following lemma : 
\begin{lemme}
	Let $V$ be a rational $\FI$-module with $t_0 V$ and $t_1 V$ finite. Then the map $(V_{n})_{S_{n+1}} \to (V_{n+1})_{S_n}$ is an isomorphism for $n \geq \max(t_0 V,t_1 V)$. 
\end{lemme}
\begin{proof}
	By \cite{Church_Ellenberg_2017}, there exists free $\FI$-modules $E$,$F$ with $t_0(F) \leq t_0(V)$ and $t_0(E) \leq \max\{t_0(V),t_1(V)\}$ and an exact sequence $E \to F \to V \to 0$. By \cite[Proposition 3.1.7]{Church_Ellenberg_2015} , in the notation of \cite[Section 6.3]{Church_Ellenberg_2015} we have
	$E \preccurlyeq ( 0, \max\{t_0(V),t_1(V)\} )$ and $F \preccurlyeq (0, t_0(V))$. Now applying  \cite[Proposition 6.3.2]{Church_Ellenberg_2015} to the complex $E \to F \to 0$ yields
	$V \preccurlyeq ( \max\{t_0(V),t_0(V)\}, t_0(V) )$,
	meaning 
	$\mrm{injdeg}(V) \leq \max\{t_0(V),t_1(V)\}$ and $\mrm{surjdeg}(V) \leq t_0(V)$
	and hence
	$\mrm{stabdeg}(V) \leq \max\{t_0(V),t_1(V)\}$. This implies that the multiplicity of the trivial representation is constant for $n\geq \max(t_0,t_1)$ and hence $V_n^{S_n} \to V_{n+1}^{S_{n+1}}$ is an iso for $n \geq \max(t_0(V),t_1(V))$.
\end{proof}
Combining this with \Cref{FI_PSub}, we get :
\begin{coroll}
When $M$ is closed, the rational cohomology groups $H^d(Sub_{\ol e_n}(P \times n, M), \bb Q)$ stabilize for $n \geq 2d+2$ if $p \leq \frac{m-3} 2$ and $n \geq 4d+6$ if $p = \frac{m-2} 2$. 	
\end{coroll}
This proves the third part of \Cref{mainresult}.
\subsection*{Open questions}
\begin{enumerate}
	\item The bounds we obtain are very close to those of Palmer, raising the question of whether there is a logical connection between the proof developed in this paper and Palmer's strategy in \cite{Palmer_2021}.
	\item In this paper we prove $\FI$-stability from high cartesianness of the cubes of embeddings. Could a dual strategy work to prove the high cartesianness of specific $n$-cubes of geometric objects ? One could hope use the techniques developed in the field of representation stability to prove that some $n$-cubes of stable homotopy types are highly cartesian in this way. What is lacking so far is a converse of \Cref{propergoingdown}. A Postnikov tower argument gives a converse of \Cref{propergoingdown} but for connective $\FI$-spectra and not co$\FI$. 
	\item In particular, Miller and Wilson \cite{Miller_Wilson_2019} show that the cubes of cochains of configurations spaces on a surface are highly cocartesian. This cannot be shown using Blakers-Massey theorem, but follows from Goodwillie and Klein's higher excision estimates. Could representation stability allow to show better connectivity ranges for stable homotopy types of spaces of embeddings ? In particular, is the cube of embeddings of unlinked circles inside a $3$-manifold $M$ highly cartesian ? Is it the case for the cube of stable homotopy types ?

	\item Work by \cite{Kupers_Miller_2016} show that the $\bb F_\ell$-homology of configuration spaces in a closed manifold is actually $\ell$-periodic. We wonder if we can improve part $2$ of \Cref{mainresult} to get $T = \ell$ as well. One approach to this would be to understand if the chain-level map $C_*(\conf_{n} M, \bb F_\ell) \to C_*(\conf_{n+p} M, \bb F_\ell)$ defined in \cite{Himes_2024} generalize to moduli spaces of submanifolds. This is not clear to the author.
\end{enumerate}

\section{Stability for symmetric and pure diffeomorphism groups of manifolds}
Let $M$ be an oriented manifold and a preferred embedding $P \times n \to \mrm{int}(M)$ as in the previous section. The \emph{symmetric diffeomorphism group} $\mathrm{SDiff}(M,P \times n)$ is the subgroup of orientation-preserving diffeomorphisms of $M$ that preserve the image of $P \times n$ \emph{setwise} and fix the eventual boundary of $M$. The \emph{pure diffeomorphism group} $\mathrm{PDiff}(M,P)$ is the subgroup of $\mathrm{SDiff}(M,P)$ that fixes each copy of $P$ setwise. In other words it is the kernel of the evaluation map $\mathrm{SDiff} \to \Sigma_n$. When $M$ has a boundary component, homological stability for the symmetric diffeomorphism groups in the case where $P$ is a point has been obtained by Tillmann in \cite{Tillmann_2016} and generalized by Palmer in \cite{Palmer_2020} for more general $P$. Regarding the pure diffeomorphism group, one could not expect stability but rather representation stability as for the case of the pure braid group. Jimenez Rolland \cite{JimenezRolland_2015}, \cite{JimenezRolland_2019} has studied this representation stability property for these pure diffeomorphism groups of manifolds and pure mapping class groups of surfaces. In this section we connect these results and show in a unified proof how one is the consequence of the other. This allows to generalize to closed manifolds what was done for manifolds with boundary in \cite{Tillmann_2016} and obtain as before periodicity instead of stabilization when working over $\bb F_\ell$, and a rational stability result. 

 We start by constructing a co$\FI$-classifying space $\mathrm{BPDiff}(M,P \times \bullet)$. As for the case of pure braid groups, the action of the symmetric group exists only at the level of the classifying space and not at the group level.
\begin{defn}
Let $\mathsf{pdiff}(M,P)$ be the following co$\FI$-topological groupoid. To a finite set $S$ we define a groupoid $\mathsf{pdiff}(M,P)_S$ whose objects are points of $\mrm{PSub}_{\ol e_n}(P \times S, M)$ where $S$ is a finite set. The space of morphisms between two objects is the space of diffeomorphisms $\varphi : M \to M$ such that $\varphi(\iota(P \times S)) = \iota'(P \times S)$ preserving the labeling by $S$, $\iota$ and $\iota'$ denoting embeddings realizing $P \times S$ as submanifolds of $M$. The collection of all the $\mathsf{pdiff}(M)_S$ has an obvious co$\FI$-structure given by forgetting the copies of $P$ and relabeling.  For each $S$, $\msf{pdiff}(M,P)_S$ is a connected groupoid, because of the isotopy extension theorem. Objectwise, $B\mathsf{pdiff}(M,P)_S$ has the homotopy type of $\mrm{BPDiff}(M, P \times S)$. We can thus define the co$\FI$-space $\mathrm{BPDiff}(M,P \times \bullet)$ to be the classifying space of the topological groupoid $\mathrm{pdiff}(M)_\bullet$. 
\end{defn}
\begin{lemme}
	For any finite set $S$ of $M$, we have a natural action of $\Sigma_S$ on $\mathrm{BPDiff}(M,P\times S)$ and the homotopy quotient is $\mathrm{BSDiff}(M,P\times S)$. 
\end{lemme}
\begin{proof}
	The symmetric group action is inherited from the co$\FI$-structure on $\mathrm{PDiff}(M,P)$. There is a fibration sequence
	$$ \mathrm{PDiff}(M,P) \to \mathrm{SDiff}(M,P) \to \Sigma_n \to \mathrm{BPDiff}(M,P) \to \mathrm{BSDiff}(M,P)$$ 
	coming from the evaluation map $\mrm{SDiff}(M,P) \to \Sigma_n$ which proves the second point. 

\end{proof}
From this fact, homological stability/ periodicity for symmetric diffeomorphims will be a direct consequence of representation stability for the pure diffeomorphism groups. 
\begin{lemme}
	Fix a submanifold $N$ as the image of the preferred embedding $ e_n$ in $\mrm{Sub}_{\ol e_n}(P \times S, M)$. Evaluation of a diffeomorphism of $M$ on this submanifold gives a fiber sequence
$$ \mrm{PDiff}(M, P \times S) \to \mrm{Diff}(M) \xto{ev_N} \mrm{PSub}_{\ol e_n}(P \times S, M)$$
\end{lemme}
\begin{proof}
	First of all, we need to show that the action of $\mrm{Diff}(M)$ on $\mrm{PSub}(P \times S, M)$ preserves the connected component $\mrm{PSub}_{\ol e_n}(P \times S,M)$. Remember that we consider only orientation-preserving diffeomorphisms. The labelled submanifold $N$, by construction, is contained in an arbitrary small ball $D$ in $M$. Any diffeomorphism of $M$ is isotopic to one fixing the small ball $D$, because there is a fiber sequence
	$$ \mrm{Diff}(M, D) \to \mrm{Diff}(M) \to \mrm{Emb}^{+}(D, M)$$
	Where $\mrm{Emb}^{+}(D, M)$ is the space of orientation-preserving embeddings of a $m$-dimensional ball inside $M$. We have $\mrm{Emb}^{+}(D,M)\simeq Fr^{+}(M)$ the oriented frame bundle of $M$ which is connected.  The long exact sequence in homotopy shows that $\pi_0 \mrm{Diff}(M,D) \to \pi_0 \mrm{Diff}(M)$ is surjective. By classical arguments \cite[Theorem 1]{Cerf_1962}, the evaluation map $\mrm{Diff}(M) \to \mrm{PSub}_{\ol e_n}(P \times S, M)$ is a fibration and the fiber is exactly the diffeomorphisms that fix $N$ setwise, and fix each connected component. 
\end{proof}

\begin{lemme}
	The total fiber of the $n$-cube $\mrm{BPDiff}(M,P \times \bullet)$ is the  total fiber of the $n$-cube of pure submanifolds $S \mapsto \mrm{PSub}(P \times S, M)$ on $M$ given by forgetting the copies.
\end{lemme}

\begin{proof}
	This is probably well-known. 
	The fiber sequence of the previous lemma deloops twice (see \cite[ Proof of Theorem 3.1]{Tillmann_2016}) and we get a fiber sequence 
	$$  \mrm{PSub}_{\ol e_n}(P \times S, M) \to \mrm{BPDiff}(M,P \times S) \to \mrm{BDiff}(M)$$
	total fibers commute with fiber sequences so we get
	$$ \tfib\mrm{PSub}_{\ol e_n}(P \times S, M) \to \tfib \mrm{BPDiff}(M,S) \to \tfib \mrm{BDiff}(M)$$
	As the rightmost term is a point (the cube is constant and hence cartesian), we get the result.
\end{proof}
Notice that we can apply the same reasoning for $M$ non-orientable,  considering this time all diffeomorphisms. In that case the argument works because $\mrm{Emb}(D, M) \simeq Fr(M)$ the (unoriented) frame bundle of $M$ which is connected when $M$ is nonorientable. We will write in both cases $\mrm{Diff}$ (resp $\mrm{PDiff}, \mrm{SDiff}$) to denote either orientation-preserving diffeomorphisms when $M$ is oriented, or all diffeomorphisms for $M$ nonorientable. \\
Applying the same arguments as in part $1$, we get the following result. 
\begin{coroll}
	If $M$ is a smooth, closed, connected manifold of dimension $d$ greater than $3$ and $P$ a closed manifold of handle dimension $p$, equipped with a preferred embedding $e : P \to M$, the cohomology of the pure diffeomorphism groups $\mrm{PDiff}(M,P \times \bullet)$ has finite presentation degree as an $\FI$-module : 
	\begin{enumerate}
	\item{$t_0 H^d \mrm{PDiff}(M,P \times \bullet)\leq \begin{cases} 
			2d+2 , & \text{if } p \leq \frac{m-3} 2, \\[6pt]
			4d+6& \text{if } p = \frac{m-2} 2 \\[6pt]
		\end{cases}$}
	\item{$t_1 H^d \mrm{PDiff}(M,P \times \bullet) \leq \begin{cases} 
			2d+3 , & \text{if } p \leq \frac{m-3} 2, \\[6pt]
4d+7& \text{if } p = \frac{m-2} 2 \\[6pt]
		\end{cases}$\\
	Moreover if $M$ has nonempty boundary, the cohomology groups are $\FI\sharp$ modules and 
	
			\[
t_0 H^d(\mathrm{PDiff}(M, P \times \bullet)) \leq 
\begin{cases}
	d + 1 & \text{if } p \leq \frac{m - 3}{2}, \\
	2d + 3 & \text{if } p = \frac{m - 2}{2}.
\end{cases}
\] }
\end{enumerate}
\end{coroll}
\begin{proof}
	For $M$ closed, we just need to apply the results of section $2$. When $M$ has boundary,  one can extend diffeomorphisms by the identity on a collar neighborhood of $\partial M$, as in \cite[Section 4]{Palmer_2020}. This gives a homotopy $\FI \sharp$-structure to the collection $\mrm{BPDiff}(M, P\times \bullet)$, which allows to use the sharper \Cref{fisharpgoingdown}.
\end{proof}
This result improves bounds obtained by Jimenez Rolland \cite[Theorem 1.1]{JimenezRolland_2019}, with a different proof. 
Following the same path as in section 3, this implies stability/periodicity for cohomology of symmetric diffeomorphism groups.
\begin{coroll}
	For $d \geq 3$, the cohomology of symmetric diffeomorphism groups of a $d$-dimensional manifold $M$ stablilizes if $M$ has nonempty boundary with $\bb Z$ coefficients, periodic if $M$ is closed, with $\bb F_\ell$ coefficients, and stabilizes with rational coefficients. All the ranges are the same as those obtained in part $1$ for spaces of submanifolds. 
\end{coroll}

\bibliographystyle{alpha}
\bibliography{bibliEmbFI.bib}
\end{document}